\documentclass[11pt, reqno]{article}
\usepackage{amsmath,amssymb,amsfonts,amsthm}
\usepackage{color}

\usepackage[colorlinks=true,linkcolor=blue,citecolor=black,urlcolor=black,pdfborder={0 0 0}]{hyperref}
\usepackage{graphicx}
\usepackage{cleveref}

\usepackage{bbm}

\usepackage{caption}
\usepackage{thmtools}
\usepackage{commath}
\usepackage{mathrsfs}

\usepackage{enumitem}

\crefname{theorem}{Theorem}{Theorems}
\crefname{thm}{Theorem}{Theorems}
\crefname{lemma}{Lemma}{Lemmas}
\crefname{lem}{Lemma}{Lemmas}
\crefname{remark}{Remark}{Remarks}
\crefname{prop}{Proposition}{Propositions}
\crefname{defn}{Definition}{Definitions}
\crefname{corollary}{Corollary}{Corollaries}
\crefname{conjecture}{Conjecture}{Conjectures}
\crefname{question}{Question}{Questions}
\crefname{chapter}{Chapter}{Chapters}
\crefname{section}{Section}{Sections}
\crefname{figure}{Figure}{Figures}

\theoremstyle{plain}
\newtheorem{thm}{Theorem}[section]
\newtheorem*{thm*}{Theorem}
\newtheorem{lemma}[thm]{Lemma}

\newtheorem{corollary}[thm]{Corollary}
\newtheorem{prop}[thm]{Proposition}

\theoremstyle{definition}

\theoremstyle{remark}

\numberwithin{equation}{section}

\renewcommand{\P}{\mathbb P}
\newcommand{\E}{\mathbb E}

\newcommand{\R}{\mathbb R}
\newcommand{\Z}{\mathbb Z}
\newcommand{\N}{\mathbb N}

\newcommand{\F}{\mathfrak F}

\newcommand{\cF}{\mathcal F}
\newcommand{\cG}{\mathcal G}

\newcommand{\cT}{\mathcal T}

\newcommand{\sA}{\mathscr A}
\newcommand{\sB}{\mathscr B}
\newcommand{\sC}{\mathscr C}
\newcommand{\sD}{\mathscr D}
\newcommand{\sE}{\mathscr E}

\newcommand{\sI}{\mathscr I}

\newcommand{\sM}{\mathscr M}

\newcommand{\sW}{\mathscr W}

\newcommand{\fF}{\mathfrak F}
\newcommand{\fG}{\mathfrak G}

\usepackage[margin=1.15in]{geometry}
\usepackage{extarrows}
\usepackage{titling}
\usepackage{commath}
\usepackage{mathtools}
\usepackage{titlesec}
\newcommand{\eps}{\varepsilon}
\usepackage{bbm}
\usepackage{setspace}
\usepackage{enumitem}
\usepackage{tikz}

\usepackage{cite}

\newcommand{\symdif}{\hspace{.1em}\triangle\hspace{.1em}}
\newcommand{\bP}{\mathbf{P}}

\pretitle{\begin{flushleft}\Large}
\posttitle{\par\end{flushleft}\vskip 0.5em}
\preauthor{\begin{flushleft}}
\postauthor{
\\
\vspace{0.5em}
\footnotesize{Statslab, DPMMS, University of Cambridge.\\ Email: 
\href{mailto:t.hutchcroft@maths.cam.ac.uk}{t.hutchcroft@maths.cam.ac.uk}}
\end{flushleft}}
\predate{\begin{flushleft}}
\postdate{\par\end{flushleft}}
\title{\bf Indistinguishability of collections of trees in the uniform spanning forest}

\renewenvironment{abstract}
 {\par\noindent\textbf{\abstractname.}\ \ignorespaces}
 {\par\medskip}

\titleformat{\section}
  {\normalfont\large \bf}{\thesection}{0.4em}{}

\author{{\bf Tom Hutchcroft}}
\begin{document}

\date{\small{\today}}

\maketitle


\begin{abstract}
We prove the following indistinguishability theorem for $k$-tuples of trees in the uniform spanning forest of $\Z^d$: Suppose that $\sA$ is a property of a $k$-tuple of components that is stable under finite modifications of the forest. Then either every $k$-tuple of distinct trees has property $\sA$ almost surely, or no $k$-tuple of distinct trees has property $\sA$ almost surely. This generalizes the indistinguishability theorem of the author and Nachmias (2016), which applied to individual trees. Our results apply more generally to any graph that has the Liouville property and for which every component of the USF is one-ended.

\end{abstract}

\setstretch{1.1}

\section{Introduction}

The \textbf{uniform spanning forests} are infinite-volume analogues of uniform spanning trees, and can be defined for any connected, locally finite graph $G$ as weak limits of the uniform spanning trees on certain finite graphs derived from $G$. These limits can be taken with respect to two extremal boundary conditions, leading to the \textbf{free uniform spanning forest} (FUSF) and \textbf{wired uniform spanning forest} (WUSF). For many graphs, such as the hypercubic lattice $\Z^d$,  the FUSF and WUSF coincide and we speak simply of the USF.
Being an open condition, connectivity is not necessarily preserved by taking weak limits, and it is possible for the USF to be disconnected. Indeed, Pemantle \cite{Pem91} proved that the USF of $\Z^d$ is a.s.\ connected if and only if $d\leq 4$. More generally, Benjamini, Lyons, Peres, and Schramm \cite{BLPS,lyons2003markov} proved that the WUSF of an infinite graph $G$ is a.s.\ connected if and only if two independent random walks on $G$ intersect infinitely often a.s.

This disconnectivity leads us to consider the following natural question: If the USF is disconnected, how different can the different components of the forest be? 
For instance, is it possible that some are recurrent while others are transient? Similarly, could it be possible that there exists a single ``thick" component that occupies a positive density of space, while all other components are ``thin" and have zero spatial density?
Benjamini, Lyons, Peres, and Schramm \cite{BLPS} conjectured the following answer to questions of this nature: If $G=(V,E)$ is \emph{transitive} and \emph{unimodular}
  (e.g., if $G$ is a Cayley graph of a finitely generated group) and $\F$ is either the WUSF or FUSF of $G$, then the components of $\F$ are \emph{indistinguishable} from each other. This means that for every measurable set $\sA \subseteq \{0,1\}^E$ of subgraphs of $G$ that is invariant under the automorphisms of $G$, either every component of $\F$ is in $\sA$ a.s.\ or none of the components of $\F$ are in $\sA$ a.s. This conjecture followed earlier work of Lyons and Schramm \cite{LS99}, who proved an analogous theorem in the context of Bernoulli percolation.  The conjecture regarding the USF was verified (in slightly greater generality) by the author and Nachmias \cite{HutNach2016a}, while partial progress was also made in the independent work of Tim\'ar \cite{timar2015indistinguishability}. 
  In the setting of Bernoulli percolation, various extensions and generalizations of the Lyons-Schramm theorem have subsequently been obtained by Aldous and Lyons \cite{AL07}, Martineau \cite{martineau2012ergodicity}, and Tang \cite{tang2018heavy}. 
  Besides their intrinsic probabilistic interest, such indistinguishability theorems have also found applications in ergodic theory, see e.g.\ \cite{MR3621428,TDprep,gaboriau2009measurable}. 


In this paper, we are interested in a form of indistinguishability that holds not only for individual components of the forest, but rather for arbitrary finite collections of components. 
Our results are motivated by our work with Yuval Peres  on the adjacency structure of the trees in the USF of $\Z^d$ \cite{hutchcroft2017component}, in which we use  the results of this paper as a zero-one law to boost positive-probability statements to almost-sure statements.
%
%
 A remarkable feature of the results we obtain is that, unlike in \cite{LS99,HutNach2016a}, we do not require any kind of homogeneity assumptions on $G$ (such as transitivity or unimodularity), nor do we require any kind of automorphism-invariance type assumptions on the properties we consider.
 Rather, the primary assumption we make on $G$ is that it is \emph{Liouville}, i.e., does not admit non-constant bounded harmonic functions. We also restrict attention to \emph{tail} properties of tuples of components, which are stable under finite modifications of the forest. Heuristically, one can think of our proof as lifting the tail triviality of the random walk (which is equivalent to the Liouville property) to indistinguishability of trees in the USF, which is itself a strong form of tail triviality.

Let us now give the definitions required to state our main theorems. 
Let $G$ be a graph, and let $k\geq 1$. We equip the set $\Omega_{k}(G):=\{0,1\}^E\times V^k$ with its product topology and associated Borel $\sigma$-algebra. We think of this set as the set of subgraphs of $G$ rooted at an ordered $k$-tuple of vertices. A measurable set $\sA \subseteq \Omega_k(G)$ is said to be a $k$-\textbf{component property} if 
\[ (\omega,(u_i)_{i=1}^k)\in \sA \Longrightarrow (\omega,(v_i)_{i=1}^k)\in \sA \,\,\,
\begin{array}{l}
\text{for all } (v_i)_{i=1}^k \in V^k \text{ such that $u_i$ is} \\\text{connected to $v_i$ in $\omega$ for each $i=1,\ldots,k$}.
\end{array}
\]
In other words, $\sA$ is a $k$-component property if it is invariant under replacing the root vertices with other root vertices from within the same components. We call $\sA$ a \textbf{multicomponent property} if it is a $k$-component property for some $k$. 
Given a $k$-component property $\sA$, we say that a $k$-tuple of components $K_1,\ldots,K_k$ of a configuration $\omega \in \{0,1\}^E$ \textbf{has property} $\sA$ if $(\omega,(v_i)_{i=1}^k) \in \sA$ whenever $u_1,\ldots,u_k$ are vertices of $G$ such that $u_i \in K_i$ for every $1 \leq i \leq k$. 


Given a vertex $v$ of $G$ and a configuration $\omega \in \{0,1\}^E$, let $K_\omega(v)$ denote the connected component of $\omega$ containing $v$. We say that a $k$-component property $\sA$ is a \textbf{tail} $k$-component property if 
\[ (\omega,(v_i)_{i=1}^k)\in \sA \Longrightarrow (\omega',(v_i)_{i=1}^k)\in \sA \,\,\,
\begin{array}{l}
\forall \omega' \in \{0,1\}^E \text{ such that } \omega \symdif \omega' \text{ is finite and }\\
 K_\omega(v_i)\symdif K_{\omega'}(v_i) \text{ is finite for every $ i =1,\ldots,k$,}
\end{array}
\]
where $\symdif$ denotes the symmetric difference. 
In other words, tail multicomponent properties are stable under finite modifications to $\omega$ that result in finite modifications to each of the components of interest $K_\omega(v_1),\ldots,K_\omega(v_k)$. For example, if $G$ is locally finite then the set $\sA_{k}$ of $(\omega,(v_i)_{i=1}^k)$ such that $K_\omega(v_i)$ contains infinitely many vertices adjacent to $K_\omega(v_j)$ for every $1 \leq i < j \leq k$  is a tail $k$-component property for each $k \geq 2$. 

We now state our result in the case $G=\Z^d$. The general result is given below.


\begin{thm}
\label{thm:indist}
Let $d\geq 5$ and let $\F$ be the uniform spanning forest of $\Z^d$.
 Then for each $k\geq 1$ and each tail $k$-component property $\sA \subseteq \Omega_k(\Z^d)$, either every $k$-tuple of distinct connected  components of $\F$ has property $\sA$ almost surely or no $k$-tuple of distinct connected  components of $\F$ has property $\sA$ almost surely.
\end{thm}

The general form of our result will require the underlying graph to be \emph{Liouville}, i.e., not admitting any non-constant bounded harmonic functions. Note that if $G$ is Liouville then its free and wired uniform spanning forests coincide \cite[Theorem 7.3]{BLPS}, so that we may speak simply of the USF of $G$. Our proof will also require that every component of the USF of $G$ is \emph{one-ended} almost surely. 
Here, an infinite graph is said to be 
\textbf{one-ended} if deleting any finite set of vertices from the graph results in at most one infinite connected component. In particular, a 
tree is one-ended if it does not contain a simple bi-infinite path. It is known that every component of the wired uniform spanning forest is one-ended almost surely in several large classes of graphs \cite{Pem91,BLPS,AL07,LMS08,H15,HutNach2016b,hutchcroft2015interlacements}, including all transitive graphs not rough-isometric to $\Z$ \cite{LMS08}. In particular, \cref{thm:indist} applies to all of the transitive graphs of polynomial growth that are studied in \cite{hutchcroft2017component}. (The condition that $G$ is one-ended in the following theorem is in fact redundant, being implied by the other hypotheses.)

\begin{thm}
\label{thm:indistgeneral}
Let $G=(V,E)$ be a infinite, one-ended, connected, locally finite graph, and suppose that $G$ is Liouville.  
Let $\F$ be the uniform spanning forest of $G$, and suppose further that every component of $\F$ is one-ended almost surely.
 Then for each $k\geq 1$ and each tail $k$-component property $\sA \subseteq \Omega_k(G)$, either every $k$-tuple of distinct connected  components of $\F$ has property $\sA$ almost surely or no $k$-tuple of distinct connected  components of $\F$ has property $\sA$ almost surely.
\end{thm}



The special case of \cref{thm:indist} concerning a single component (i.e., $k=1$)
is essentially equivalent to \cite[Theorem 4.5]{BeKePeSc04}, which was not phrased in terms of indistinguishability. \cref{thm:indist} is also closely related to \cite[Theorem 1.20]{HutNach2016a}, which implies in particular that components of the wired uniform spanning forest of any transitive graph are indistinguishable from each other  by automorphism-invariant tail properties.
See \cref{sec:closing} for a discussion of how \cref{thm:indistgeneral} can fail in the absence of the assumption that $G$ is Liouville, even if we require that $G$ is a Cayley graph and $\sA$ is automorphism invariant.

We will assume that the reader is familiar with the definition of the uniform spanning forest and with Wilson's algorithm, referring them to e.g.\ \cite{LP:book} for background otherwise. 

\section{Proof}
\label{sec:indist}

Indistinguishability is closely related to tail-triviality. 
Let $\Omega$ be a measurable space, let $I$ be a countable set, and let $\Omega^I = \{(\omega_i)_{i\in I} : \omega_i \in \Omega \text{ for every $i \in I$}\}$ be  equipped with the product $\sigma$-algebra $\cF$. 
For each subset $J$ of $I$, 
we define $\cF_J$ to be the sub-$\sigma$-algebra of $\cF$ of events depending only on $(\omega_i)_{i\in J}$. The \textbf{tail $\sigma$-algebra} of $\Omega^I$ is defined to be the intersection 
\[\mathcal{T}=\bigcap \left\{\cF_J : I \setminus J \text{ is finite} \right\}.\] An $\Omega^I$-valued random variable $A=(A_i)_{i\in I}$ is said to be \textbf{tail-trivial} if it has probability either zero or one of belonging to any set in the tail $\sigma$-algebra $\cT$. 

The following was proven by Benjamini, Lyons, Peres, and Schramm \cite[Theorem 8.3]{BLPS}, generalizing a result of Pemantle \cite{Pem91}.

\begin{thm}
\label{thm:USFtailtriviality}
Let $G=(V,E)$ be an infinite, connected, locally finite graph, and let $\F \in \{0,1\}^E$ be either the free or wired uniform spanning forest of $G$. Then $\F$ is tail-trivial.
\end{thm}

In particular, if the USF of $G$ has only one component a.s.\ then \cref{thm:indistgeneral} is implied by \cref{thm:USFtailtriviality}. Thus, it suffices to prove \cref{thm:indistgeneral} in the case that the USF of $G$ has more than one component with positive probability, in which case $G$ must be transient. 

Recall that the \textbf{lazy random walk} on a graph is the random walk that stays where it is with probability $1/2$ at each step, but otherwise chooses a uniform edge emanating from its current location just as the usual random walk does. Note that we can use lazy random walks instead of simple random walks when sampling the WUSF of a graph using Wilson's algorithm, since doing so does not affect the law of the resulting forest. This will be useful to us thanks to the following well-known theorem due to Blackwell \cite{Blackwell55} and Derriennic \cite{Derriennic80}; see also \cite[Corollary 14.13 and Theorem 14.18]{LP:book}.
 (Laziness is used in this theorem to avoid parity issues.)

\begin{thm}
Let $G=(V,E)$ be an infinite, connected, locally finite graph, and let $X \in V^\N$ be a lazy random walk on $G$ started at some vertex $v$. Then $G$ is Liouville if and only if $X$ is tail-trivial.  
\end{thm}

It will also be useful for us to use the following well-known equivalence between tail-triviality and asymptotic independence: See e.g.\ \cite[Proposition 10.17]{LP:book} or \cite[Proposition 7.9]{GeorgiiBook}. 
\begin{prop}
\label{lem:tailtriviality}
Let $\Omega$ be a measurable space, let $I$ be a countable set, let $A=(A_i)_{i \in I}$ be an $\Omega^I$-valued random variable with law $\P$, and let $(K_n)_{n\geq0}$ be an increasing sequence of finite subsets of $I$ such that $\bigcup_{n\geq0} K_n = I$. Then $A$ is tail-trivial if and only if for every event 
 $\sA \subseteq \Omega^I$ for which $\P(A \in \sA)>0$ we have that
\begin{equation}
\label{eq:tailtrivlem0} 
\lim_{n\to\infty} \sup \left\{ \left |\P\left(A \in \sA  \text{ and } A \in \sB \right) - \P\left( A \in \sA\right)\P\left( A \in \sB\right)\right| : \sB \in \cF_{I\setminus K_n} \right\} = 0
\end{equation}
and hence that
\begin{equation}
\label{eq:tailtrivlem} 
\lim_{n\to\infty} \sup \left\{ \left |\P\left(A \in \sB  \mid A \in \sA \right) - \P\left( A \in \sB\right)\right| : \sB \in \cF_{I \setminus K_n} \right\} = 0.
\end{equation}
In other words, $A$ is tail-trivial if and only if the total variation distance between the distribution of $(A_i)_{i \in I \setminus K_n}$ and the conditional distribution of $(A_i)_{i \in I \setminus K_n}$ given $\{A \in \sA\}$ converges to zero as $n\to\infty$ for every event $\sA \subseteq \Omega^I$.
\end{prop}

Let us note also that if $\Omega_1$, $\Omega_2$ are measurable spaces, $I$ is a countable set, $A$ is an $\Omega_1$ valued random variable and $B^1,\ldots,B^k$ are independent, tail-trivial, $\Omega_2^I$-valued random variables, then $(X_i)_{i\in I} =((A_i,(B_i^j)_{j=1}^k))_{i \in I}$ is a tail trivial $(\Omega_1 \times \Omega_2^k)^I$-valued random variable.

Our first step towards \cref{thm:indist} is the following lemma. Given $\mathbf{u}=(u_1,\ldots,u_k)$, we write $\sW(\mathbf{u})$ for the event that the vertices $u_1,\ldots,u_k$ are all in distinct components of $\F$.

\begin{lemma}
\label{lem:indistlem}
Let $G=(V,E)$ be an infinite, transient, Liouville graph,  let $\F$ be the uniform spanning forest of $G$, and suppose that every component of $\F$ is one-ended almost surely. Let $\mathbf{u}=(u_1,\ldots,u_k)$ be a $k$-tuple of vertices of $G$,
and let $\mathbf{X}=(X^1,\ldots,X^k)$ be a $k$-tuple of independent lazy random walks on $G$, independent of $\F$, such that $\mathbf{X}_0=\mathbf{u}$. 
If $\P(\sW(\mathbf{u}))>0$, then for every 
tail $k$-component property $\sA$ we have that
\vspace{0.2cm}
\begin{equation}
\vspace{0.2cm}
 \P\left((\F,\mathbf{u}) \in \sA \mid  \sW(\mathbf{u}) \right) = \lim_{m\to\infty} \P\left((\F,\mathbf{X}_m) \in \sA\right).
\label{eq:indist3}
\end{equation}
In particular, the right hand limit exists.
\end{lemma}


Before beginning the proof of this lemma, let us note the following. Suppose that $G$ is an infinite, transient, Liouville graph whose USF is disconnected with positive probability. Then whenever $X$ and $Y$  are independent lazy random walks on $G$, we have by \cite[Theorem 9.2]{BLPS}
that $X$ and $Y$ intersect only finitely often with positive probability. Since the event that both walks are transient and intersect only finitely often is a tail event, we deduce that $X$ and $Y$ intersect only finitely often almost surely. It follows that if $X^1, \ldots, X^k$ are independent lazy random walks on $G$, then
 \begin{equation}
 \label{eq:liouvilleintersections}
 \lim_{m\to\infty} \P\left( (X^i_{n+m})_{n \geq 0} \text{ and } (X^j_{n+m})_{n \geq 0} \text{ intersect for some $1 \leq i < j \leq k$}\right)=0. 
 \end{equation}
 In particular, it follows from \cite[Theorem 9.4]{BLPS} that the USF of $G$ has infinitely many connected components almost surely.

The proof of \cref{lem:indistlem} will also apply the following simple measure-theoretic lemma.

\begin{lemma}
\label{lem:measuretheory}
Let $(X_i)_{i\geq 1}$ and $X$ be random variables defined on a shared probability space $(\Omega,\P)$ and taking values in a locally compact  Hausdorff space $\mathbb{X}$. Let $(B_i)_{i\geq 1} \subseteq \Omega$ and $B_i \subseteq \Omega$ be measurable with $\P(B)>0$. Suppose further that the following hold:
\begin{enumerate}
\item $X_i$ and $X$ have the same distribution for every $i\geq 1$.
\item $X_i$ converges to $X$ in probability as $i\to\infty$.
\item $\P(B_i \symdif B) \to 0$ as $i\to \infty$.
\end{enumerate}
Then $\P( X \in A \mid B) = \lim_{i\to\infty}\P( X_i \in A \mid B_i )$ for every Borel set $A \subseteq \mathbb{X}$.
\end{lemma}

\begin{proof}

By \cite[Theorem 3.14]{MR924157}, for every $\eps>0$ there exists a continuous function $f_\eps:\mathbb{X}\to\R$ such that 
$
\E\left[|f_\eps(X)-\mathbbm{1}(X\in A)|\right] \leq \eps$.
We have by the triangle inequality that
\begin{align*}
|\P(X\in A,B) - \P(X_i\in A, B_i)|&\leq 
\E\left[ |\mathbbm{1}(X\in A) \mathbbm{1}(B)- \mathbbm{1}(X_i\in A) \mathbbm{1}(B_i)|\right]\\&\leq \E\left[ |\mathbbm{1}(X\in A) - f_\eps(X)| \mathbbm{1}(B)\right] + \E\left[ f_\eps(X)|\mathbbm{1}(B) -\mathbbm{1}(B_i)|\right]\\
&\hspace{1cm}+\E\left[ |f_\eps(X)-f_\eps(X_i)| \mathbbm{1}(B_i)\right]
+\E\left[ |f_\eps(X_i)-\mathbbm{1}(X_i\in A)| \mathbbm{1}(B_i)\right]\\
&\leq 2 
\E\left[ |\mathbbm{1}(X\in A) - f_\eps(X)|\right] + \E\left[ f_\eps(X)|\mathbbm{1}(B) -\mathbbm{1}(B_i)|\right]\\
&\hspace{1cm}+\E\left[ |f_\eps(X)-f_\eps(X_i)| \right],
\end{align*}
where we used the fact that $X$ and $X_i$ have the same distribution in the third inequality. 
Applying the dominated convergence theorem we deduce that
\[
\limsup_{i\to\infty}
|\P(X\in A,B) - \P(X_i\in A, B_i)| \leq 2 \E\left[ |\mathbbm{1}(X\in A) - f_\eps(X)|\right] \leq 2\eps.
\]
The claim follows since $\eps>0$ was arbitrary.
\end{proof}

\begin{proof}[Proof of \cref{lem:indistlem}]

Let $(v_i)_{i\geq0}$ be an enumeration of the vertices of $G$ and let $( Y^{i,m} )_{i\geq 0, m\geq 0}$ be a collection of independent lazy random walks, independent of $\mathbf{X}$ and of $\F$, such that $Y^{i,m}$ is started at $v_i$ for every $i\geq 0$ and $m \geq 0$.
Let $\fG$ be a  uniform spanning forest of $G$ sampled using Wilson's algorithm, beginning with the walks $X^1,\ldots, X^k$, and then using the walks $Y^{0,0},Y^{1,0},\ldots$.  Similarly, for each $m\geq 1$, let $\fG_m$ be a  uniform spanning forest of $G$ sampled using Wilson's algorithm, beginning with the walks $(X^1_{n+m})_{n \geq 0},\ldots, (X^k_{n+m})_{n\geq 0}$ and then using the walks $Y^{0,m},Y^{1,m},\ldots$. 
We clearly have that
$\left( \fG_m,\, \mathbf{X}_m\right)$ and $\left(\F,\, \mathbf{X}_m \right)$ have the same distribution 
for every $m\geq0$. 
Unlike $\F$, the forests $\fG_m$ are \emph{not} independent of the random walks $\mathbf{X}$.


For each $m \geq 0$, we will define a sequence of forests $\F_{m,R}$ which interpolate between $\F$ and $\fG_m$. 
The \textbf{future} of a vertex $v$ in $\F$, denoted $\operatorname{fut}_\F(v)$, is defined to be the set of vertices on the unique infinite simple path starting at $v$ in $\F$, including $v$ itself. The \textbf{past} of a vertex $v$ in $\F$ is defined to be the set of all vertices $u$ such that $v$ is in the future of $u$. 
For each integer $R\geq 0$, let $\F_R$ be the subgraph of $\F$ induced by the set
\[\bigcup\left\{ \operatorname{fut}_\F(v) : v \in V \setminus B(u_1,R)\right\}, \]
%
where $B(u_1,R)$ denotes the graph-distance ball of radius $R$ around $u_1$ in $G$.
For each $R \geq r \geq 0$, let $\sC_{r,R}$ be the event that $\F_R$ does not intersect the set $\bigcup_{i=1}^k B(u_i,r)$. A vertex $v$ of $G$ is contained in the forest $\F_R$ if and only if its past intersects $V \setminus B(u_1,R)$. Since every component of $\F$ is  one-ended almost surely, the past of each vertex of $G$ is finite almost surely. We deduce that $\bigcap_{R\geq0} \F_R=\emptyset$ almost surely and hence that
$\lim_{R\to\infty}\P\left(\sC_{r,R}\right)=1$
for every $r\geq 0$.

Now, for each $m\geq 0$ and $R \geq 0$, we define a forest $\F_{m,R} = \bigcup_{i\geq0} \F_{m,R}^i$, where the forests $\F_{m,R}^i$ are defined recursively as follows. Let $\F_{m,R}^0=\F_R$.
For each $1 \leq i \leq k$, given $\fF^{i-1}_{m,R}$,  stop the random walk $(X^i_{n+m})_{n\geq 0}$ when it first visits the set of vertices already included in $\fF^{i-1}_{m,R}$. Take the loop-erasure of this stopped path, and
and let $\F_{m,R}^{i}$ be the union of $\F_{m,R}^i$ with this loop-erased path.  Similarly, if $i > k$, stop the random walk $(Y^{i-k-1,m}_{n})_{n\geq 0}$ when it first visits the set of vertices already included in $\fF^{i-1}_{m,R}$, take the loop-erasure of this stopped path, and
and let $\F_{m,R}^{i}$ be the union of $\F_{m,R}^{i-1}$ with this loop-erased path. We refer to this procedure as \textbf{completing the run} of Wilson's algorithm. It follows from \cite[Lemma 4.1]{HutNach2016a} that each of the forests $\F_{m,R}$ is distributed as the uniform spanning forest of $G$. (Indeed, we can think of the forests $\F_{m,R}$ as being sampled using  Wilson's algorithm, except that we are choosing which vertices to start our random walks from using a well-ordering of the vertex set  that is not an enumeration.)

It is not hard to see that
$\F_{m,R}$ converges to $\fG_m$
almost surely as $R \to \infty$ (with respect to the product topology on $\{0,1\}^E$). Indeed, it follows from the transience of $G$ that the loop-erasure of $(X^1_{n+m})_{n\geq 0}$ stopped when it first hits $\F_R$ converges almost surely to the loop-erasure of the unstopped walk as $R\to\infty$, and applying a similar argument inductively we deduce that $\F^i_{m,R}$ converges almost surely as $R\to\infty$ to the forest generated by the first $i$ steps of the application of Wilson's algorithm used to generate $\fG_m$; the claim that $\F_{m,R}$ converges to $\fG_m$
almost surely follows by taking $i$ to infinity. 

For each $m\geq 0$ and $R\geq0$, let $\sW_m$ be the event that the vertices $\{X^i_m : 1\leq i \leq k\}$ are all in different components of $\fG_{m}$, and let 
 $\sB_{m,R}$ to be the event that the vertices $\{ X^i_m : 1 \leq i \leq k\}$ are all in different components of the forest $\F_{m,R}\setminus \F_R$. Thus, the event $\sB_{m,R}$ occurs if and only if for each $1 \leq i \leq k$ the walk $(X^i_{n+m})_{n\geq0}$ first hits the set of vertices included in $\F^{i-1}_{m,R}$ at a vertex of $\F_R$.
It is easily seen that $\lim_{R\to\infty}\P(\sW_m \symdif \sB_{m,R})=0$ for each $m\geq 0$.

  Now, for each $m\geq 0,$ $R\geq 0$ and $1 \leq \ell \leq k$, let $\tau_\ell(m,R)$ be the first time after time $m$ that the walk $X^\ell$ hits the set of vertices included in  $\F_R$. Write $\mathbf{\tau}(m,R)=(\tau_\ell(m,R))_{\ell=1}^k$ and $\mathbf{X}_{\mathbf{\tau}(m,R)}=(X^\ell_{\tau_\ell(m,R)})_{\ell=1}^k$. 
  On the event $\sB_{m,R}$, the vertex $X^i_m$ is connected in $\F_{m,R}$ to the vertex $X^i_{\tau_i(m,R)}$. Since $\sA$ is a tail property,  we deduce that
  \[
\sB_{m,R} \cap \Bigl\{
(\F_{m,R},\mathbf{X}_m) \in \sA 
\Bigr\} = 
\sB_{m,R} \cap \Bigl\{
\bigl(\F_{R},\mathbf{X}_{\mathbf{\tau}(m,R)}\bigr) \in \sA 
\Bigr\}\]
up to a null set. 
%
%
Moreover, observe that $\tau_\ell(m,R)=\tau_\ell(0,R)$ for every $1 \leq \ell \leq k$ and every $m \leq r$ on the event $\sC_{r,R}$, so that if $r\geq m$ then 
  \begin{multline*}
\sC_{r,R}\cap \sB_{0,R}\cap\sB_{m,R} \cap \Bigl\{
(\F_{m,R},\mathbf{X}_m) \in \sA 
\Bigr\}\\ = 
\sC_{r,R}\cap \sB_{0,R}\cap \sB_{m,R} \cap \Bigl\{
\bigl(\F_{R},\mathbf{X}_{\mathbf{\tau}(m,R)}\bigr) \in \sA 
\Bigr\}\\
 = \sC_{r,R}\cap \sB_{0,R}\cap\sB_{m,R} \cap \Bigl\{
(\F_{m,R},\mathbf{u}) \in \sA 
\Bigr\}
\end{multline*}
up to null sets, and taking probabilities we have that
\begin{equation*}
\P\left( 
(\F_{m,R},\mathbf{u}) \in \sA 
 \mid \sC_{r,R}\cap \sB_{0,R}\cap\sB_{m,R} \right)
=
\P\left( 
(\F_{m,R},\mathbf{X}_m) \in \sA 
 \mid \sC_{r,R}\cap \sB_{0,R}\cap\sB_{m,R} \right).
\end{equation*}
Using \cref{lem:measuretheory} to take the limit as $R \to \infty$ on both sides, we obtain that
\begin{equation*}
\P((\fG_0,\mathbf{u}) \in \sA \mid  \sW_0,\, \sW_m) 
= \P((\fG_m,\mathbf{X}_m) \in \sA \mid  \sW_0,\, \sW_m).
\end{equation*}
for every $m \geq 0$.
The event $\sW_m$ is contained in the event that none of the random walks $(X^i_{n+m})_{n\geq0}$ intersect each other. If $k=1$, then $\sW_m$ trivially has probability one for every $m \geq 0$. Otherwise, $k>1$ and our assumption that $\P(\sW)>0$ implies that $\F$ is disconnected with positive probability, so that the event $\sW_m$ has probability converging to $1$ as $m\to\infty$ by \eqref{eq:liouvilleintersections}. In either case, we deduce that
\begin{equation}
\P((\fG_0,\mathbf{u}) \in \sA \mid  \sW_0) = \lim_{m\to\infty} 
 \P((\fG_m,\mathbf{X}_m) \in \sA \mid  \sW_0). 
\label{eq:indist1}
\end{equation}
%
%
%
In particular, the right-hand limit exists.

Since $G$ is Liouville, the sequence of random variables $((\mathbf{X}_{n+m})_{n\geq 0})_{m\geq 0}$ is tail-trivial. 
Moreover, for each $m\geq 0$ the forest $\fG_m$ is conditionally independent  given the random variables $(\mathbf{X}_{n+m})_{n\geq 0}$ of the walks $(\mathbf{X}_n)_{n=0}^m$ and the forests $(\fG_i)_{i=1}^{m-1}$, and it is easily deduced that the sequence of random variables
$\left(\fG_m,\, (\mathbf{X}_{n+m})_{n\geq0}  \right)_{m \geq0}$
is also tail-trivial. 
Applying \cref{lem:tailtriviality} we deduce that
\begin{equation}
\lim_{m\to\infty}\left|
 \P((\fG_m,\mathbf{X}_m) \in \sA \mid  \sW_0)-\P((\fG_m,\mathbf{X}_m) \in \sA)\right|  =0.
\label{eq:indist2}
\end{equation}
The claim follows by combining \eqref{eq:indist1} and \eqref{eq:indist2} and using that $(\F,\mathbf{X}_m)$ and $(\fG_m,\mathbf{X}_m)$ are equidistributed.
\end{proof}

\begin{corollary}
\label{lem:constantprobs}
Let $G=(V,E)$ be an infinite, transient, Liouville graph,  let $\F$ be the uniform spanning forest of $G$, and suppose that every component of $\F$ is one-ended almost surely. 
Then for each tail $k$-component property $\sA$ there exists a constant $P(\sA)$ such that 
\vspace{0.2cm}
\begin{equation}
\vspace{0.2cm}
 \P\left((\F,\mathbf{u}) \in \sA \mid  \sW(\mathbf{u}) \right) = P(\sA).
\label{eq:indist1000}
\end{equation}
for every $\mathbf{u}\in V^k$ with $\P(\sW(\mathbf{u}))>0$. 
\end{corollary}

\begin{proof}
Let $\sA$ be a tail $k$-component property.
Let $\mathbf{u}=(u_1,\ldots,u_k)$ and $\mathbf{u}'=(u_1',\ldots,u'_k)$ be two $k$-tuples of vertices of $G$ such that $\P(\sW(\mathbf{u})),\P(\sW(\mathbf{u}'))>0$. Let $\mathbf{X}=(X^1,\ldots,X^k)$ be a $k$-tuple of independent lazy random walks, independent of $\F$, with $\mathbf{X}_0=\mathbf{u}$. Let $M=\max d(u_i,u_i')$, and let $\sM$ be the event that $\mathbf{X}_M=\mathbf{u}'$, which is easily seen to have positive probability. It follows by \cref{lem:indistlem},  the Liouville property, \cref{lem:tailtriviality}, and the Markov property of the lazy random walk that
\begin{align*}
\P\left( (\F,\mathbf{u})\in \sA \mid \sW(\mathbf{u})\right)
&=\lim_{m\to\infty} \P\left((\F,\mathbf{X}_m)\in \sA \right)=\lim_{m\to\infty} \P\left((\F,\mathbf{X}_m)\in \sA \mid \sM \right)\\
 &= \P\left( (\F,\mathbf{u}')\in \sA \mid \sW(\mathbf{u}')\right).
\end{align*}
as claimed.
\end{proof}

It remains to prove that $P(\sA)\in \{0,1\}$ for every tail $k$-component property $\sA$.

Our next goal is to establish a conditional version of \cref{lem:indistlem}.
Let $r\geq 1$. By a slight abuse of notation, write $\F \cap B(u_1,r)$ to denote the set of edges of $\F$ that have both endpoints in the graph-distance ball $B(u_1,r)$ of radius $r$ around $u_1$, and $B(u_1,r) \setminus \F$ to denote the set of edges that have both endpoints in the ball $B(u_1,r)$ and are not contained in $\F$. 
For each $r\geq 1$, let $\cG_r$ be the $\sigma$-algebra generated by $\F \cap B(u_1,r)$. 
Similarly, for each $R\geq 1$, let $\cG^R$ be the $\sigma$-algebra generated by the restriction of $\F$ to the \emph{complement} of the ball $B(u_1,R)$.

\begin{lemma}
\label{lem:indistlem2}
Let $G=(V,E)$ be an infinite, one-ended, transient, Liouville graph,  let $\F$ be the uniform spanning forest of $G$, and suppose that every component of $\F$ is one-ended almost surely. Let $\mathbf{u}=(u_1,\ldots,u_k)$ be a $k$-tuple of vertices of $G$,
and let $\mathbf{X}=(X^1,\ldots,X^k)$ be a $k$-tuple of independent lazy random walks on $G$, independent of $\F$, such that $\mathbf{X}_0=\mathbf{u}$. 
If $\P(\sW(\mathbf{u}))>0$, then for every 
tail $k$-component property $\sA$ we have that
\vspace{0.2cm}
\begin{equation}
\vspace{0.2cm}
 \P\left((\F,\mathbf{u}) \in \sA \mid \cG_r,\,  \sW(\mathbf{u}) \right) = \lim_{m\to\infty} \P\left((\F,\mathbf{X}_m) \in \sA \mid \cG_r \right)  \quad \text{ a.s.}
\label{eq:indist3}
\end{equation}
In particular, the right hand limit exists almost surely.
\end{lemma}

Before beginning the proof, let us recall that the Liouville property can equivalently be defined in terms of the triviality of \emph{invariant} events, rather than tail events. The \textbf{invariant $\sigma$-algebra} of $V^\N$ is defined to be the set of all measurable sets $\sI \subseteq V^\N$ such that $(v_i)_{i\geq 1}\in \sI$ if and only if $(v_{i+1})_{i\geq 1}\in \sI$. A graph is Liouville if and only if the lazy random walk has probability either zero or one of belonging to any set in the invariant $\sigma$-algebra \cite[Corollary 14.13]{LP:book}. 

\setstretch{1.25}

\begin{proof}

Write $\sW=\sW(\mathbf{u})$. 
Let $A \subset E$ be such that $\F \cap B(u_1,r)=A$ with positive probability, and let $B$ be the set of edges that have both endpoints in $B(u_1,r)$ but are not in $A$.  
Let $\widehat G = (\widehat V, \widehat E)$ be the graph obtained from $G$ by contracting every edge in $A$ and deleting every edge in $B$. Note that, since $G$ is one-ended and every component of $\F$ is almost surely infinite, every two vertices of $G$ are almost surely connected by a path in $G$ that does not use any edges of $B(u_1,r)\setminus \F$, and it follows that $\widehat G$ is connected. 
  %
  Let $\widehat{\F}$ a wired uniform spanning forest of $\widehat G$ independent of $\F$. 
   By the spatial Markov property of the uniform spanning forest (see e.g.\ \cite[Section 2.2.1]{HutNach2016b}), the conditional distribution of $\F$ given $\F \cap B(u_1,r)=A$ coincides with the distribution of $\widehat \F \cup A$ (after appropriate identification of edges).

%
Let $\pi:  V \to \widehat V$ be the function sending each vertex of $V$ to its image following the contraction, and for each $v \in \widehat V$ let $\pi^{-1}(v)$ be an arbitrarily chosen vertex of $G$ such that $\pi(\pi^{-1}(v))=v$. Let $\pi^{-1}(\mathbf{v}) =(\pi^{-1}(v_1),\ldots,\pi^{-1}(v_k))$ for each $\mathbf{v}\in V^k$ and define
\[\widehat{\hspace{0cm}\sA} = \left\{(\omega,\mathbf{v}) \in  \{0,1\}^{\widehat E} \times \widehat V^k : \left(\omega \cup A , \pi^{-1}(\mathbf{v})\right) \in \sA  \right\},\]
which does not depend on the arbitrary choices used to define $\pi^{-1}$ since $\sA$ is a multicomponent property. It is easily verified that $\widehat{\hspace{0cm}\sA}$ is a tail $k$-component property, and that
\begin{equation}
\P\left((\F,\mathbf{u}) \in \sA \mid \F \cap B(u_1,r)=A,\, \sW\right) = \P\left((\widehat {\F} , \pi(\mathbf{u})) \in \widehat{\hspace{0cm}\sA} \mid \widehat{\hspace{0cm}\sW}\right),
\end{equation}
where $\widehat{\hspace{0cm}\sW}$ is the event that $\pi(u_1),\ldots,\pi(u_k)$ are all in different components of $\widehat \F$.

Let $\widehat{\mathbf{X}}=(\widehat{X}^\ell)_{\ell=1}^k$ be independent lazy random walks on $\widehat G$ that are conditionally independent of $\F$ and $\widehat \F$ given $\cG_r$ and satisfy
$\widehat{\mathbf{X}}_0=\pi(\mathbf{u})$.
Let $T_\ell$ and $\widehat T_\ell$ be the last times that the walks $ X^\ell$ and $\widehat X^\ell$ visit $B(u_1,r)$ and $\pi(B(u_1,r))$ respectively, and write $\mathbf{T}=(T_1,\ldots,T_k)$ and $\widehat{\mathbf{T}}=(\widehat T_1,\ldots,\widehat T_k)$. Define $\mathbf{X}_{\mathbf{T}+1}:=(X^\ell_{T_\ell+1})_{\ell = 1}^k$ and $\widehat{\mathbf{X}}_{\widehat{\mathbf{T}}+1}:=(\widehat{X}^\ell_{\widehat{T}_\ell+1})_{\ell = 1}^k$. 
Observe that, since $\widehat G$ is connected, the supports of the random variables $\mathbf{X}_{\mathbf{T}+1}$ and $\pi^{-1}(\widehat{\mathbf{X}}_{\widehat{\mathbf{T}}+1})$ are both equal to the set of vertices $v\in V\setminus B$ for which the random walk started at $v$ has a positive probability not to hit $B(u_1,r)$. Similarly,
the 
support of the random variable $(\mathbf{T},\mathbf{X}_{\mathbf{T}+1})$ is contained in the support of $(\widehat{\mathbf{T}},\pi^{-1}(\widehat{\mathbf{X}}_{ \widehat{\mathbf{T}}+1}))$.
Furthermore, for every $\mathbf{t}\in \N^k$ and $\mathbf{v}=(v_1,\ldots,v_k)\in V^k$ such that 
$\mathbf{T}=\mathbf{t}$ and $\mathbf{X}_{\mathbf{T}+1} = \mathbf{v}$ with positive probability, 
 we have the equality of conditional distributions
\begin{multline}
\label{eq:conditionallaws}
\bigl(\text{Law of } \bigl(\pi^{-1}(\widehat{\mathbf{X}}_{\widehat{\mathbf{T}}+n})\bigr)_{n\geq 1} \text{ given $\widehat{\mathbf{T}}=\mathbf{t}$ and  $\pi^{-1} (\widehat{\mathbf{X}}_{\widehat{\mathbf{T}}+1} )= \mathbf{v}$}\bigr)\\
 = \left(\text{Law of } (\mathbf{X}_{\mathbf{T}+n})_{n\geq 1} \text{ given $\mathbf{T}=\mathbf{t}$ and $\mathbf{X}_{\mathbf{T}+1} = \mathbf{v}$}\right). 
 \end{multline}
Similarly, for every $\mathbf{v}=(v_1,\ldots,v_k)\in V^k$ such that 
$\mathbf{T}=\mathbf{t}$ and $\mathbf{X}_{\mathbf{T}+1} = \mathbf{v}$ with positive probability, 
 we have the equality of conditional distributions
 \begin{multline}
\label{eq:conditionallaws2}
\bigl(\text{Law of } \bigl(\pi^{-1}(\widehat{\mathbf{X}}_{\widehat{\mathbf{T}}+n})\bigr)_{n\geq 1} \text{ given  $\pi^{-1}(\widehat{\mathbf{X}}_{\widehat{\mathbf{T}}+1}) = \mathbf{v}$}\bigr)
 = \left(\text{Law of } (\mathbf{X}_{\mathbf{T}+n})_{n\geq 1} \text{ given  $\mathbf{X}_{\mathbf{T}+1} = \mathbf{v}$}\right). 
 \end{multline}
  Indeed, both sides of both \eqref{eq:conditionallaws} and \eqref{eq:conditionallaws2} are equal to the law of a $k$-tuple of independent lazy random walks on $G$, started at the vertices $(v_\ell)_{\ell=1}^k$ and conditioned not to return to $B(u_1,r)$. 

We deduce from \eqref{eq:conditionallaws} that $\widehat G$ is Liouville: If $\sI \subseteq \widehat V^{\N}$ is an invariant event, then $\sI' = \{ (v_n)_{ n \geq 0} \in V^{\N} : (\pi(v_n))_{n \geq 0} \in \sI\}$ is also an invariant event, and hence that 
\[\P((X^1_n)_{n\geq 1} \in \sI')=\P((X^1_{T_1+n})_{n\geq 1} \in \sI') \in \{0,1\}\] since $G$ is Liouville, from which it follows that
$\P((X^1_{T_1+n})_{n\geq 1} \in \sI' \mid X^1_{T_1}) \in \{0,1\}$
almost surely. 
The equality of the supports of $X^1_{T_1+1}$ and of $\widehat X ^1_{\widehat T_1 +1 }$ and of the conditional laws \eqref{eq:conditionallaws} then implies that $\P(\widehat X^1 \in \sI)$ is also equal to either zero or one, and since $\sI$ was arbitrary it follows that $\widehat G$ is Liouville as claimed. Thus, applying \cref{lem:indistlem} to both $G$ and $\widehat G$ yields that 
\begin{equation}
 \P\left((\F,\mathbf{u}) \in \sA \mid  \sW \right) = \lim_{m\to\infty} \P\left((\F,\mathbf{X}_m) \in \sA\right)
 \label{eq:indistlem2eq4}
\end{equation}
and
\begin{align}
\P\Bigl(\bigl(\widehat\F,\pi(\mathbf{u})\bigr) \in \widehat{\hspace{0cm}\sA} \mid  \widehat{\hspace{0cm}\sW}\Bigr)
= 
\lim_{m\to\infty}
\P\Bigl(\bigl(\widehat\F, \widehat{\mathbf{X}}_m\bigr) \in \widehat{\hspace{0cm}\sA} \Bigr). 
\label{eq:indistlem2eq5}
\end{align}

Let $\mathbf{t}$ and $\mathbf{v}$ be such that $\mathbf{T}=\mathbf{t}$ and $
\mathbf{X}_{\mathbf{T}+1} = \mathbf{v}$ with positive probability. Conditioning on $\F$ and applying 
 \cref{lem:tailtriviality} we deduce that 
\begin{equation}
\lim_{m\to\infty}\bigl|\P\left(( \F,\mathbf{X}_m) \in \sA \mid \F,\,  
\mathbf{T} = \mathbf{t},\, 
\mathbf{X}_{\mathbf{T}+1} = \mathbf{v}\right)
-
\P\left(( \F,\mathbf{X}_m) \in \sA \mid \F \right)\bigr| =0
\end{equation}
almost surely, and hence that
\begin{multline}
\label{eq:tv1}
\lim_{m\to\infty}\bigl|\P\left(( \F,\mathbf{X}_m) \in \sA \mid  \F \cap B(u_1,r) =A,\,
\mathbf{T} = \mathbf{t},\, 
\mathbf{X}_{\mathbf{T}+1} = \mathbf{v}\right)\\
-
\P\left(( \F,\mathbf{X}_m) \in \sA  \mid \F \cap B(u_1,r) =A \right)\bigr| =0.
\end{multline}
Applying a similar analysis to $\widehat G$ yields that
\begin{equation}
\lim_{m\to\infty}\bigl|\P\bigl(( \widehat \F, \widehat{\mathbf{X}}_m) \in \widehat{\hspace{0cm}\sA} \mid   
\widehat{\mathbf{T}} = \mathbf{t},\, 
\pi^{-1}(\widehat{\mathbf{X}}_{\widehat{\mathbf{T}}+1}) = \mathbf{v}\bigr)- \P\bigl(( \widehat \F, \widehat{\mathbf{X}}_m) \in \widehat{\hspace{0cm}\sA}  \bigr)\bigr| =0.
\label{eq:tv2}
\end{equation}
On the other hand, the spatial Markov property of the USF and the equality of conditional laws \eqref{eq:conditionallaws} implies that
\begin{multline}
\P\left(( \F,\mathbf{X}_m) \in \sA \mid \F \cap B(u_1,r) =A,\,  
\mathbf{T} = \mathbf{t},\, 
\mathbf{X}_{\mathbf{T}+1} = \mathbf{v}\right)\\
=
\P\bigl(( \widehat{\F}, \widehat{\mathbf{X}}_m) \in \widehat{\hspace{0cm}\sA} \mid   
\widehat{\mathbf{T}} = \mathbf{t},\, 
\pi^{-1}(\widehat{\mathbf{X}}_{\widehat{\mathbf{T}}+1}) = \mathbf{v}\bigr)
\label{eq:tv3}
\end{multline}
for every $m\geq 1+\max_{1\leq i \leq k} t_i$. 
Together, \eqref{eq:tv1}, \eqref{eq:tv2}, and \eqref{eq:tv3} imply that
\begin{equation}
\lim_{m\to\infty}\Bigl| \P\left(( \F,\mathbf{X}_m) \in \sA \mid \F \cap B(u_1,r) =A \right) - \P\bigl((\widehat \F,\widehat{\mathbf{X}}_m) \in \widehat{\hspace{0cm}\sA} \bigr)\Bigr| = 0
\label{eq:indistlem2eq2}
\end{equation}
almost surely, 
which yields the claim when combined with  \eqref{eq:indistlem2eq4} and \eqref{eq:indistlem2eq5}. \qedhere

\end{proof}

\setstretch{1.1}

We are now ready to complete the proof of \cref{thm:indist}. 

\begin{proof}[Proof of \cref{thm:indist}]
By \cref{lem:constantprobs}, it remains to prove only that $P(\sA)\in \{0,1\}$ for every tail $k$-component property $\sA$. 
We continue to use the notation of \cref{lem:indistlem,lem:indistlem2}. 
Since $\bigcup_{r\geq0} \cG_r$ generates the product $\sigma$-algebra of $\{0,1\}^E$, it suffices to prove that
\begin{equation}
\label{eq:indistclaim1}
 \P\left((\F,\mathbf{u}) \in \sA \mid \cG_r,\, \sW\right) = \P\left((\F,\mathbf{u}) \in \sA \mid \sW\right) =P(\sA)
\quad \text{ a.s.}\end{equation}
for every tail $k$-component property $\sA$, every $\mathbf{u}=(u_1,\ldots,u_k)\in V^k$ with $\P(\sW(\mathbf{u}))>0$, and every $r \geq 1$.
By \cref{lem:indistlem,lem:indistlem2}, to prove \eqref{eq:indistclaim1} it suffices to prove that 
\begin{equation}
\label{eq:indistclaim3}
\lim_{m\to\infty} \P\left((\F,\mathbf{X}_m) \in \sA \right) =  \lim_{m\to\infty}\P\left((\F,\mathbf{X}_m) \in \sA \mid \cG_r\right) \quad \text{ a.s.}
\end{equation}

Let $\sD_{m,R}$ be the event that the future of $X^\ell_m$ in $\F$ is contained in the complement of $B(u_1,R)$ for every $1\leq \ell \leq k$, and let $\sE_{m,R} = \{ (\F,\mathbf{X}_m) \in \sA\} \cap \sD_{m,R}$. In particular, $X^\ell_m \in V \setminus B(u_1,R)$ for every $1\leq \ell\leq k$ on the event $\sE_{m,R}$. 
 Moreover, since $\F$ is one-ended almost surely and $G$ is transient, we have that $\P(\sD_{m,R}) \to 1$ as $m\to\infty$ and hence that 
\[\lim_{m\to\infty} \left| \P(\sE_{m,R}) - \P\left(\left(\F,\mathbf{X}_m\right) \in \sA\right)\right|=0 \] 
for every $R\geq 1$. Thus, there exists a sequence $m(R)$ growing sufficiently quickly that 
\begin{equation}
\label{eq:EmRR}
\lim_{R\to\infty} \Bigl| \P(\sE_{m(R),R}) - \P\bigl((\F,\mathbf{X}_{m(R)}) \in \sA\bigr)\Bigr|=0. 
\end{equation}

Let $A$ be a set of edges such that $\P(\F \cap B(u_1,r) = A ) >0$. Since $\sA$ is a tail property and every component of $\F$ is one-ended almost surely, the event $\sE_{m,R}$
is measurable (up to a null set) with respect to the $\sigma$-algebra generated by $\cG^R$ and $\mathbf{X}_m$. Thus, by tail-triviality of the uniform spanning forest, conditioning on $\mathbf{X}$ and applying \cref{lem:tailtriviality} yields that
\begin{equation*}
\lim_{R\to\infty} \left|\P\left(\left\{\F \cap B(u_1,r)=A\right\} \cap \sE_{m(R),R} \mid \mathbf{X}\right) -  \P\left(\F \cap B(u_1,r)=A\right)\P\left(\sE_{m(R),R} \mid \mathbf{X} \right)\right| = 0 
\end{equation*}
almost surely. Taking expectations over $\mathbf{X}$, we deduce that
\begin{equation*}
\lim_{R\to\infty} \left|\P\left(\left\{\F \cap B(u_1,r)=A\right\} \cap \sE_{m(R),R}\right)-  \P(\F \cap B(u_1,r)=A)\P\left(\sE_{m(R),R}\right)\right| = 0.
\end{equation*}
Dividing through by 
$\P(\F \cap B(u_1,r) = A)$ then yields that
\begin{align*}\lim_{R\to\infty}
\left|\P\left( \sE_{m(R),R} \mid  \F \cap B(u_1,r) =A \right) - \P\left( \sE_{m(R),R}\right)\right| =0
\end{align*}
and applying \eqref{eq:EmRR} we deduce that
\begin{equation*}\lim_{R\to\infty}
\left|\P\left( \left(\F,\mathbf{X}_{m(R)}\right) \in \sA  \mid  \F \cap B(u_1,r) =A \right) - \P\left( \left(\F,\mathbf{X}_{m(R)}\right) \in \sA\right)\right| =0.
\end{equation*}
Since $A$ was arbitrary, the claimed equality \eqref{eq:indistclaim3} follows.
\end{proof}


\section{The Liouville property is necessary}

\label{sec:closing}

Suppose $G$ is a non-Liouville graph such that every component of the USF of $G$ is one-ended almost surely. 
Let $\sA$ be a non-trivial invariant event for the random walk on $G$ and let $h(v)=\bP_v(\sA)$ be the associated bounded harmonic function.
It follows from the martingale convergence theorem that $h(X_n) \to \mathbbm{1}(A)$ almost surely as $n\to\infty$ whenever $X$ is a random walk on $G$. Thus, by Wilson's algorithm, almost surely for every tree of $\F$, the value of $h$ converges as we move progressively higher up the tree. Thus, we can assign a value of either zero or one to each tree of $\F$, and the value of the tree is a tail component property. Moreover, it is easy to see that there must be trees with both values zero and one a.s. This shows that, without the Liouville condition, \cref{thm:indist} always fails even in the case $k=1$. 


If $G$ is a transitive graph, it is natural to consider tail multicomponent properties that are invariant under the automorphisms of $G$. In this case, \cite[Theorem 1.20]{HutNach2016a} implies the indistinguishability of individual components of the WUSF by automorphism invariant tail properties, corresponding to the case $k=1$ of \cref{thm:indist}. The question of whether a similar result holds for larger $k$ seems to depend on the symmetries of $G$. For example, if $G$ is the $7$-regular triangulation, then we can consider the circle packing of $G$ into the unit disc (see e.g.\ \cite{Rohde11} and references therein), which is unique up to M\"obius transformations and reflections. 
Every tree in the wired uniform spanning forest of $G$ has a unique limit point in the unit circle under this embedding \cite{BS96a}, and we can define a tail $4$-component property by asking whether, given some $4$-tuple of trees, the cross-ratio of their limit points has absolute value greater than one. Some $4$-tuples of trees in the WUSF will satisfy this property while others will not, so that it is possible to distinguish $4$-tuples of distinct trees in the WUSF via tail $4$-component properties. On the other hand, the unit circle is the Poisson boundary of $G$ \cite{ABGN14} and hence, intuitively, all the tail information about a collection of trees should be contained in their collection of limit points. Since the group of M\"obius transformations of the unit disc acts $3$-transitively on the unit circle, it is plausible that it should be impossible to distinguish $3$-tuples of distinct trees in the WUSF of this graph via multicomponent properties. 

\subsection*{Acknowledgments}
This work took place while the author was an intern at Microsoft Research, Redmond.


\footnotesize{
\bibliographystyle{abbrv}
\bibliography{unimodular}

\begin{thebibliography}{10}

\bibitem{AL07}
D.~Aldous and R.~Lyons.
\newblock Processes on unimodular random networks.
\newblock {\em Electron. J. Probab.}, 12:no. 54, 1454--1508, 2007.

\bibitem{ABGN14}
O.~Angel, M.~T. Barlow, O.~Gurel-Gurevich, and A.~Nachmias.
\newblock Boundaries of planar graphs, via circle packings.
\newblock {\em Ann. Probab.}, 44(3):1956--1984, 2016.

\bibitem{BeKePeSc04}
I.~Benjamini, H.~Kesten, Y.~Peres, and O.~Schramm.
\newblock Geometry of the uniform spanning forest: transitions in dimensions
  {$4,8,12,\dots$}.
\newblock {\em Ann. of Math. (2)}, 160(2):465--491, 2004.

\bibitem{BLPS}
I.~Benjamini, R.~Lyons, Y.~Peres, and O.~Schramm.
\newblock Uniform spanning forests.
\newblock {\em Ann. Probab.}, 29(1):1--65, 2001.

\bibitem{BS96a}
I.~Benjamini and O.~Schramm.
\newblock Harmonic functions on planar and almost planar graphs and manifolds,
  via circle packings.
\newblock {\em Invent. Math.}, 126(3):565--587, 1996.

\bibitem{Blackwell55}
D.~Blackwell.
\newblock On transient {M}arkov processes with a countable number of states and
  stationary transition probabilities.
\newblock {\em Ann. Math. Statist.}, 26:654--658, 1955.

\bibitem{Derriennic80}
Y.~Derriennic.
\newblock Quelques applications du th\'eor\`eme ergodique sous-additif.
\newblock In {\em Conference on {R}andom {W}alks ({K}leebach, 1979)
  ({F}rench)}, volume~74 of {\em Ast\'erisque}, pages 183--201, 4. Soc. Math.
  France, Paris, 1980.

\bibitem{gaboriau2009measurable}
D.~Gaboriau and R.~Lyons.
\newblock A measurable-group-theoretic solution to von neumann’s problem.
\newblock {\em Inventiones mathematicae}, 177(3):533--540, 2009.

\bibitem{GeorgiiBook}
H.-O. Georgii.
\newblock {\em Gibbs measures and phase transitions}, volume~9 of {\em De
  Gruyter Studies in Mathematics}.
\newblock Walter de Gruyter \& Co., Berlin, second edition, 2011.

\bibitem{MR3621428}
M.~Gheysens and N.~Monod.
\newblock Fixed points for bounded orbits in {H}ilbert spaces.
\newblock {\em Ann. Sci. \'Ec. Norm. Sup\'er. (4)}, 50(1):131--156, 2017.

\bibitem{H15}
T.~Hutchcroft.
\newblock Wired cycle-breaking dynamics for uniform spanning forests.
\newblock {\em Ann. Probab.}, 44(6):3879--3892, 2016.

\bibitem{hutchcroft2015interlacements}
T.~Hutchcroft.
\newblock Interlacements and the wired uniform spanning forest.
\newblock {\em Ann. Probab.}, 46(2):1170--1200, 03 2018.

\bibitem{HutNach2016a}
T.~Hutchcroft and A.~Nachmias.
\newblock Indistinguishability of trees in uniform spanning forests.
\newblock {\em Probability Theory and Related Fields}, pages 1--40, 2016.

\bibitem{HutNach2016b}
T.~Hutchcroft and A.~Nachmias.
\newblock Uniform spanning forests of planar graphs.
\newblock {\em arXiv preprint arXiv:1603.07320}, 2016.

\bibitem{hutchcroft2017component}
T.~Hutchcroft and Y.~Peres.
\newblock The component graph of the uniform spanning forest: Transitions in
  dimensions $9, 10, 11,\ldots$.
\newblock {\em arXiv preprint arXiv:1702.05780}, 2017.

\bibitem{LMS08}
R.~Lyons, B.~J. Morris, and O.~Schramm.
\newblock Ends in uniform spanning forests.
\newblock {\em Electron. J. Probab.}, 13:no. 58, 1702--1725, 2008.

\bibitem{LP:book}
R.~Lyons and Y.~Peres.
\newblock {\em Probability on Trees and Networks}, volume~42 of {\em Cambridge
  Series in Statistical and Probabilistic Mathematics}.
\newblock Cambridge University Press, New York, 2016.
\newblock Available at \url{http://pages.iu.edu/~rdlyons/}.

\bibitem{lyons2003markov}
R.~Lyons, Y.~Peres, and O.~Schramm.
\newblock Markov chain intersections and the loop-erased walk.
\newblock In {\em Annales de l'institut Henri Poincar{\'e} (B) Probabilit{\'e}s
  et Statistiques}, volume~39, pages 779--791, 2003.

\bibitem{LS99}
R.~Lyons and O.~Schramm.
\newblock Indistinguishability of percolation clusters.
\newblock {\em Ann. Probab.}, 27(4):1809--1836, 1999.

\bibitem{martineau2012ergodicity}
S.~Martineau.
\newblock Ergodicity and indistinguishability in percolation theory.
\newblock {\em Enseign. Math.}, 61(3-4):285--319, 2015.

\bibitem{Pem91}
R.~Pemantle.
\newblock Choosing a spanning tree for the integer lattice uniformly.
\newblock {\em Ann. Probab.}, 19(4):1559--1574, 1991.

\bibitem{Rohde11}
S.~Rohde.
\newblock {Oded Schramm: From Circle Packing to SLE}.
\newblock {\em Ann. Probab.}, 39:1621--1667, 2011.

\bibitem{MR924157}
W.~Rudin.
\newblock {\em Real and complex analysis}.
\newblock McGraw-Hill Book Co., New York, third edition, 1987.

\bibitem{tang2018heavy}
P.~Tang.
\newblock Heavy bernoulli-percolation clusters are indistinguishable.
\newblock {\em arXiv preprint arXiv:1809.01284}, 2018.

\bibitem{timar2015indistinguishability}
A.~Tim\'ar.
\newblock Indistinguishability of the components of random spanning forests.
\newblock {\em Ann. Probab.}, 46(4):2221--2242, 2018.

\bibitem{TDprep}
R.~Tucker-Drob.
\newblock Ergodic hyperfinite subgraphs and one-ended subforests.
\newblock In preparation.

\end{thebibliography}
}
\end{document}